\documentclass[fleqn,final]{zzz}
\usepackage{mathrsfs}
\usepackage{color}
\usepackage{tikz}
\usepackage{graphicx}
\usepackage{rotating}
\usepackage[pdftex]{hyperref}

\hypersetup{
   colorlinks = true,
   urlcolor = blue,
   linkcolor = red,
   citecolor = green
}

\newcommand{\hyp}[5]{\,\mbox{}_{#1}F_{#2}\!\left(
  \genfrac{}{}{0pt}{}{#3}{#4};#5\right)}

\newcommand{\qhyp}[5]{\,\mbox{}_{#1}\phi_{#2}\!\left(
\genfrac{}{}{0pt}{}{#3}{#4};#5\right)}

\newcommand{\myref}[1]{(\ref{#1})}
\def\cprime{$'$}

\newtheorem{thm}[lemma]{Theorem}
\newtheorem{cor}[lemma]{Corollary}
\newtheorem{prop}[lemma]{Proposition}

\makeatletter
\def\eqnarray{\stepcounter{equation}\let\@currentlabel=\theequation
\global\@eqnswtrue
\tabskip\@centering\let\\=\@eqncr
$$\halign to \displaywidth\bgroup\hfil\global\@eqcnt\z@
  $\displaystyle\tabskip\z@{##}$&\global\@eqcnt\@ne
  \hfil$\displaystyle{{}##{}}$\hfil
  &\global\@eqcnt\tw@ $\displaystyle{##}$\hfil
  \tabskip\@centering&\llap{##}\tabskip\z@\cr}

\def\endeqnarray{\@@eqncr\egroup
      \global\advance\c@equation\m@ne$$\global\@ignoretrue}

\def\@yeqncr{\@ifnextchar [{\@xeqncr}{\@xeqncr[5pt]}}
\makeatother

\begin{document}

\renewcommand{\PaperNumber}{***}

\FirstPageHeading

\ShortArticleName{
Generalizations of
generating function
for basic hypergeometric
orthogonal polynomials}

\ArticleName{Generalizations of generating functions 
for basic hypergeometric orthogonal polynomials}

\Author{Howard S.~Cohl\,$^\dag\!\!\ $, Roberto S.~Costas-Santos\,$^\S\!\!\ $,
Philbert R.~Hwang\,$^\ddag{}$,
and Tanay Wakhare\,$^\ast{}$}

\AuthorNameForHeading{H.~S.~Cohl, R. S.~Costas-Santos, 
P. Hwang, T.~Wakhare}
\Address{$^\dag$~Applied and Computational Mathematics Division,
National Institute of Standards and Technology,
Mission Viejo, CA 92694, USA
\URLaddressD{
\href{http://www.nist.gov/itl/math/msg/howard-s-cohl.cfm}
{http://www.nist.gov/itl/math/msg/howard-s-cohl.cfm}
}
} 
\EmailD{howard.cohl@nist.gov} 

\Address{$^\S$~Departamento de F\'isica y Matem\'{a}ticas,
Universidad de Alcal\'{a},
c.p. 28871, Alcal\'{a} de Henares, Spain
} 
\URLaddressD{
\href{http://www.rscosan.com}
{http://www.rscosan.com}
}
\EmailD{rscosa@gmail.com} 
\Address{$^\ddag$~Department of Computer Science, 
University of Maryland, College Park, MD 20742, USA}
\EmailD{hwangphilbert@gmail.com} 
\Address{$^\ast$~Department of Mathematics, University 
of Maryland, College Park, MD 20742, USA}
\EmailD{twakhare@gmail.com} 

\ArticleDates{Received 19 April 2018 in final 
form ????; Published online ????}

\Abstract{We derive generalized generating functions 
for basic hypergeometric orthogonal polynomials by 
applying connection relations with one free parameter 
to them.  In particular, we generalize generating 
functions for the Askey-Wilson, continuous 
$q$-ultraspherical/Rogers,  little $q$-Laguerre/Wall, 
and $q$-Laguerre polynomials.
Depending on what type of orthogonality these 
polynomials satisfy, we derive corresponding definite 
integrals, infinite series, bilateral infinite 
series, and $q$-integrals.
}

\Classification{33C45, 05A15, 33C20, 34L10, 30E20}

\section{Introduction}
\label{Introduction}

In the context of generalized hypergeometric 
orthgononal polynomials H. Cohl developed in 
\cite[(2.1)]{CohlGenGegen} therein) a series 
rearrangement technique which produces a 
generalization of the generating function for 
the Gegenbauer polynomials.  We have since  
demonstrated that this technique is valid for 
a larger class of hypergeometric orthogonal 
polynomials. For instance, in \cite{Cohl12pow} 
we applied this same technique to the Jacobi 
polynomials, and in Cohl et al. \cite{CohlMacKVolk}, 
we extended this technique to many generating 
functions for the Jacobi, Gegenbauer, Laguerre, 
and Wilson polynomials.  

The series rearrangement technique combines a 
connection relation with a generating function, 
resulting in a series with multiple sums. The 
order of summations are then rearranged and 
the result often simplifies to produce a
generalized generating function whose coefficients 
are given in terms of generalized or basic 
hypergeometric functions.  This technique is 
especially productive when using connection 
relations with one free parameter, since the 
relation is most often a product of Pochhammer 
and $q$-Pochhammer symbols.

Basic hypergeometric orthogonal polynomials with more 
than one free parameter, such the Askey-Wilson 
polynomials, have multi-parameter connection relations. 
These connection relations are in general given by
single or multiple summation expressions. 
For the Askey-Wilson polynomials, the connection 
relation with four free parameters is given as a basic 
double hypergeometric series. The fact that the 
four free parameter connection coefficient for the 
Askey-Wilson polynomials is given by a double sum was
known to Askey and Wilson as far back as 1985 (see 
\cite[p. 444]{Ismail}). When our series rearrangement 
technique is applied to cases with more
than one free parameter, the resulting coefficients 
of the generalized generating function are rarely 
given in terms of a basic hypergeometric series. 
The more general problem of generalized generating 
functions with more than one free parameter requires 
the theory of multiple basic hypergeometric
series and is not treated in this paper. 

The coefficients of our derived generalized 
generating functions  are basic hypergometric 
functions.  There are many known summation 
formulae for basic hypergeometric functions (see 
for instance, \cite[Sections 17.5--17.7]{NIST:DLMF}). 
One could study the special combinations of 
parameters which allow for summability of our 
basic hypergeometric coefficients.  
However, in these cases the affect of summability 
for special parameter values, simply reduces to 
a re-expression of the original generating 
function used to generate the generalizations 
that we present. So therefore nothing new 
would then be learned by this study.

In this paper, we apply this technique to 
generalize generating functions for basic 
hypergeometric orthogonal polynomials in the 
$q$-analog of the Askey scheme 
\cite[Chapter 14]{Koekoeketal}.
These are the continuous 
$q$-ultraspherical/Rogers polynomials
(Section \ref{Rogerscontinuousqultrasphericalpolynomials}),
little $q$-Laguerre polynomials
(Setion \ref{LittleqLaguerreWallpolynomials}),
$q$-Laguerre polynomials
(Section \ref{qLaguerrepolynomials})
and the Askey-Wilson polynomials (Section 
\ref{AskeyWilsonpolynomials}).
In Section 
\ref{Definiteintegralsinfiniteseriesandqintegrals},
we have also computed new definite integrals, 
infinite series, and Jackson integrals (hereafter 
$q$-integrals) corresponding to our generalized 
generating function expansions using orthogonality 
for the studied basic hypergeometric orthogonal
polynomials.

Note that one important class of hypergeometric 
orthogonal polynomial generating functions 
which does not seem amenable to our series 
rearrangement technique are bilinear generating 
functions. The existence of an extra orthogonal 
polynomial in the generating function, produces 
multiple summation expressions via the introduction 
of connection relations for one or both of the 
polynomials with the sums being formidable to 
evaluate in closed form.

\section{Preliminaries}
Throughout the paper, we will adopt the following 
notation to indicate sequential positive and negative 
elements, in a list of elements, namely
\[
\pm a:=\{a,-a\}.
\]
If $\pm$ appears in an expression, but not in 
a list, it is to be treated as normal.
In order to obtain our derived identities, we 
rely on properties of the$q$-Pochhammer symbol 
($q$-shifted factorial).  The $q$-Pochhammer 
symbolis defined for $n\in\mathbb 
N_0:=\{0, 1, 2,\dots\}$ such that
\begin{equation}
\label{2:1}
(a;q)_0:=1,
\quad
(a;q)_n:=(1-a)(1-aq)\cdots(1-aq^{n-1}),
\end{equation}
and
\begin{equation}
(a;q)_\infty:=\prod_{n=0}^\infty (1-aq^{n}),
\label{2:2}
\end{equation}
where $0<|q|<1$, $a\in\mathbb C$.
We define the $q$-factorial as 
\cite[(1.2.44)]{GaspRah}
\[
[0]_q!:=1, \ [n]_q!:=[1]_q[2]_q\cdots [n]_q, \ n\ge 1,
\]
where the $q$-number is defined as \cite[(1.8.1)]{Koekoeketal}
\[
[z]_q:=\frac{1-q^z}{1-q}, \quad z\in \mathbb C,
\]
with $q\in\mathbb C$, $q\ne 1$.
Note that $[n]_q! = (q;q)_n / (1-q)^n$.

The following properties for the $q$-Pochhammer 
symbol  can be found in Koekoek et al. 
\cite [(1.8.7), (1.8.10-11), (1.8.14), (1.8.19), 
(1.8.21-22)]{Koekoeketal}, namely for appropriate 
values of $a$ and $k\in\mathbb N_0$,
\begin{eqnarray}
\label{2:3}
&&\hspace{-6.5cm}(a^{-1};q)_n=\frac{(-1)^n}{a^n}
q^{\binom n 2}(a;q^{-1})_n,\\[0.2cm] \label{2:4}
&&\hspace{-6.5cm}(a;q)_{n+k} = (a;q)_k(aq^k;q)_n 
= (a;q)_n(aq^n;q)_k,\\[0.2cm] \label{2:5}
&&\hspace{-6.5cm}(aq^n;q)_{k} = \frac{(a;q)_k}
{(a;q)_n}(aq^k;q)_n,\\[0.2cm]
\label{2:6}
&&\hspace{-6.5cm}(aq^{-n};q)_{k}=q^{-nk}
\frac{(q/a;q)_n}{(q^{1-k}/a;q)_n}(a;q)_k,\\[0.2cm]
\label{2:7}
&&\hspace{-6.5cm}(a;q)_{2n}=(a,aq;q^2)_n,\\[0.2cm]
\label{2:8}
&&\hspace{-6.5cm}(a^2;q^2)_n=(\pm a;q)_n.
\end{eqnarray}

Observe that by using (\ref{2:1}) and 
(\ref{2:8}), we get
\begin{eqnarray}
&&\hspace{-6.5cm}(aq^n;q)_n=\frac{(\pm 
\sqrt{a},\pm \sqrt{aq};q)_n}{(a;q)_n}. 
\label{2:9}
\end{eqnarray}


\begin{lemma}
	Let $q, \alpha,\beta\in\mathbb C$, $0<|q|<1$. Then
	\begin{equation}
	\label{2:10}
	\lim_{q\uparrow1^{-}}\frac{(q^\alpha;q)_\beta}
	{(1-q)^\beta}=(\alpha)_\beta.
	\end{equation}
\end{lemma}
\begin{proof}
	Define the $q$-gamma function $\Gamma_q$ by 
	\cite[(1.9.1)]{Koekoeketal}
	\[
	\Gamma_q(x):=\frac{(1-q)^{1-x}(q;q)_\infty}
	{(q^x;q)_\infty},
	\]
	and the arbitrary $q$-Pochhammer symbol by 
	(\ref{2:7}).
	
	Observe that, by using (\ref{2:6}), 
	if $\Re \beta<0$ then
	\begin{equation}
	(a;q)_\beta:=\frac 1{(a q^{\beta};q)_{-\beta}}.
	\label{2:11}
	\end{equation}
	Taking the previous expressions we have 
	that the arbitrary Pochhammer symbol for 
	$\Re \beta>0$ is defined naturally by
	\[
	(\alpha)_\beta:=\frac{\Gamma(\alpha+\beta)}
	{\Gamma(\alpha)},
	\]
	and if $\Re \beta<0$ then
	$(\alpha)_\beta:=1/(\alpha+\beta)_{-\beta}$.
	
	Then
	\begin{itemize}
		\item If $\alpha+\beta\in -\mathbb N_0$ then 
		the result is straightforward by definition
		since $(-n)_n=0$ and $(q^{-n};q)_n=0$ for any 
		$n\in \mathbb N_0$.
		\item If $\Re \beta>0$ then
		\[
		\lim_{q\uparrow1^{-}}
		\frac{(q^\alpha;q)_\beta}{(1-q)^\beta}
		=\lim_{q\uparrow1^{-}}
		\frac{(q^\alpha;q)_\infty}{(1-q)^\beta(q^{\alpha
				+\beta};q)_\infty}=\lim_{q\uparrow1^{-}}
		\frac{\Gamma_q(\alpha+\beta)}{\Gamma_q(\alpha)}
		=(\alpha)_\beta,
		\]
		since \cite[Section 1.9]{Koekoeketal} $\lim_{q\uparrow1^{-}}\Gamma_q(x)=\Gamma(x)$.
		\item If $\Re \beta<0$ then 
		\[
		\lim_{q\uparrow1^{-}}
		\frac{(q^\alpha;q)_\beta}{(1-q)^\beta}
		\stackrel{(\ref{2:11})}=\lim_{q\uparrow1^{-}}
		\frac{(1-q)^{-\beta}}
		{(q^{\alpha+\beta};q)_{-\beta}}
		=\lim_{q\uparrow1^{-}}\frac{(1-q)^{-\beta}
			(q^\alpha;q)_\infty}{(q^{\alpha+\beta}
			;q)_\infty}=\lim_{q\uparrow1^{-}}\frac
		{\Gamma_q(\alpha+\beta)}{\Gamma_q(\alpha)}
		=(\alpha)_\beta.
		\]
	\end{itemize}
\end{proof}


We also take advantage of the $q$-binomial 
theorem \cite[(1.11.1)]{Koekoeketal}
\[
\qhyp10a-{q,z}=\frac{(az;q)_\infty}{(z;q)_\infty},
\qquad |z|<1,
\]
where we have used \myref{2:2}.
The basic hypergeometric series, which we 
will often use, is defined as
\cite[(1.10.1)]{Koekoeketal}
\begin{equation}
{}_r\phi_s\left(
\begin{array}{c}
a_1,\ldots,a_r\\
b_1,\ldots,b_s
\end{array}
;q,z
\right):=\sum_{k=0}^\infty
\frac{(a_1,\ldots,a_r;q)_k}
{(q,b_1,\ldots,b_s;q)_k}
\left((-1)^kq^{\binom k2}\right)^{1+s-r}
z^k.
\label{2:12}
\end{equation}

Let us prove some inequalities that we will 
later use.
\begin{lemma} \label{lem:4}
Let $j\in \mathbb N$, $k,n\in \mathbb N_0$, 
$z\in \mathbb C$, $\Re u>0$, $v\ge 0$, 
and $|q|<1$. Then
\begin{eqnarray}
\label{2:13} \dfrac{(q^u;q)_j}{(1-q)^j}&\ge& 
[\Re u]_q [j-1]_q!, \\ \label{2:14} 
\dfrac{(q^u;q)_n}{(q;q)_n}&\le&  [1+n]_q^ u, \\
\label{2:15} \dfrac{(q^{v+k};q)_n}
{(q^{u+k};q)_n}&\le& \frac{[n+1]_
q^{v+1}}{[\Re (u)]_q},\\ \label{2:16} \dfrac{(q^{z+k};q)_{n-k}}{(1-q)^{n-k}}&\le&  
\frac{[n]_q!}{[k]_q!} [1+n]_q^{|z|}.
\end{eqnarray}
\end{lemma}
\begin{proof}
If $|q|<1$ then
\[
\dfrac{(q^u;q)_j}{(1-q)^j}=\prod_{k=1}^{j-1} 
\frac{1-q^{u+k-1}}{1-q}\ge \frac{1-q^u}{1-q}
\prod_{k=1}^{j-1} \frac{1-q^{k}}{1-q}\ge 
[\Re(u)]_q [j-1]_q!.
\]
This completes the proof of \eqref{2:13}. 
Choose $m\in \mathbb N_0$ such that $m\le
u \le m + 1$.  Then $q^{m+1}\le q^u$, so
\[
\dfrac{(q^u;q)_n}{(q;q)_n}=\prod_{k=0}^{n-1} 
\frac{1-q^{u+k}}{1-q^{1+k}}
\le \prod_{k=1}^n \frac{1-q^{m+k}}{1-q^{k}}
=\prod_{k=1}^m \frac{1-q^{n+k}}{1-q^{k}}
\le [n+1]_q^m\le [n+1]_q^{u}.
\]
This completes the proof of \eqref{2:14}. Without 
loss of generality we assume $u>0$. If $v\le  u$ 
then the inequality is clear, so let us assume
that $0<u<v$. Since $|q|<1$ and for $t\ge 0$,
\[
\frac{t+v}{t+u}\le \frac v u,
\]
we have
\[
\dfrac{(q^{v+k};q)_n}{(q^{u+k};q)_n}\le
\frac{(q^v)_n}{(q^u)_n}\le
\frac 1{[u]_q}\frac{(q^v)_n}{[n-1]_q!(1-q)^n}.
\]
Choose $m\in \mathbb N$ so that $m-1<v\le m$. 
Then
\[
\dfrac{(q^{v+k};q)_n}{(q^{u+k};q)_n}\le
\frac 1{[u]_q}\frac{[n]_q(q^m;q)_n}{(q;q)_n}=
\frac 1{[u]_q}\frac{[n]_q(q^n;q)_{m-1}}
{(q;q)_{m-1}}\le\frac 1{[u]_q}[n]_q 
[n+1]_q^{m-1}\le\frac 1{[u]_q} [n+1]_q^{v+1}.
\]
This completes the proof of \eqref{2:15}. 
Finally if $|z|\le 1$, then
\[
\dfrac{(q^{z+k};q)_{n-k}}{(1-q)^{n-k}}
=\prod_{\lambda=1}^{n-k}\frac{1-q^{z+k-1+
\lambda}}{1-q}\le\prod_{\lambda=1}^{n-k}
\frac{1-q^{k+\lambda}}{1-q}=\frac{[n]_q!}
{[k]_q!}.
\]
If $|z|>1$ then using \eqref{2:14}
\[
[k]_q!\dfrac{(q^{z+k};q)_{n-k}}{(1-q)^{n-k}}
\le [|z|]_q [|z|+1]_q\cdots [|z|+n-1]_q\le
[n]_q! [1+n]_q^{|z|}.
\]
Therefore all the previous formulae hold true.
\end{proof}

\section{Continuous $q$-ultraspherical/Rogers polynomials}
\label{Rogerscontinuousqultrasphericalpolynomials}

The continuous $q$-ultraspherical/Rogers 
polynomials are defined as 
\cite[(14.10.17)]{Koekoeketal}
\[
C_n(x;\beta|q):=\frac{(\beta;q)_n}
{(q;q)_n}e^{in\theta}\,{}_2\phi_1\left(
\begin{array}{c} q^{-n},\beta\\[0.2cm]
\beta^{-1}q^{1-n}
\end{array};q,q\beta^{-1}e^{-2i\theta}
\right), \qquad x=\cos \theta.
\]
By starting with generating functions for the 
continuous $q$-ultraspherical/Rogers polynomials \cite[(14.10.27--33)]{Koekoeketal}, we derive 
generalizations using the connection relation 
for these polynomials, namely
\cite[(13.3.1)]{Ismail}
\begin{equation}
\label{3:17}
C_{n}(x;\beta\,|\,q) =\frac{1}{1-\gamma}
\sum_{k=0}^{\lfloor n/2 \rfloor}
\frac{(1-\gamma q^{n-2k}) \gamma^{k}(\beta
\gamma^{-1};q)_{k}(\beta;q)_{n-k}}{(q;q)_{k}
(q\gamma;q)_{n-k}}\, C_{n-2k}(x;\gamma|q).
\end{equation}

\begin{thm} \label{theo:4}
Let $x\in[-1,1]$, $|t|<1-\beta^2$, $\beta,
\gamma\in (-1,1)\setminus\{0\},$ $|q|<1$. 
Then
\begin{eqnarray}
&&(te^{-i\theta};q)_\infty\,{}_2\phi_1
\left(\begin{array}{c}
\beta,\beta e^{2i\theta}\\ \beta^2
\end{array};q,te^{-i\theta}\right)
=\sum_{n=0}^\infty\frac{(\beta;q)_n\,q^{\binom n2}
(-\beta t)^n}{(\gamma,\beta^2;q)_n}C_n(x;\gamma|q)
\nonumber\\ &&\hspace{2cm}\times\, {}_2\phi_5\left(
\begin{array}{c} \beta\gamma^{-1},\beta q^n\\
\gamma q^{n+1},\pm \beta q^{n/2}, \pm \beta q^{(n+1)/2}
\end{array};q,\gamma(\beta t)^2 q^{2n+1}\right).
\label{gengenthm2}
\end{eqnarray}
\end{thm}
\begin{proof}
A generating function for continuous 
$q$-ultraspherical/Rogers polynomials
can be found in Koekoek et al. 
\cite[(14.10.29)]{Koekoeketal}
\begin{equation}
(te^{-i\theta};q)_\infty\,{}_2\phi_1
\left(\begin{array}{c} \beta,\beta 
e^{2i\theta}\\ \beta^2\end{array}
;q,te^{-i\theta}\right)=\sum_{n=0}^\infty
\frac{q^{\binom n2}(-\beta t)^n}{(\beta^2;q)_n}
C_n(x;\beta|q). \label{genfun141029}
\end{equation}
Start with (\ref{genfun141029}),
inserting (\ref{3:17}), shifting the 
$n$ index by $2k$, reversing the order of 
summation and using \myref{2:4} 
through \myref{2:12}, and by noting
\[
{\binom {n+2k} 2}={\binom n2}
+4{\binom k 2}+(2n+1)k.
\]
This completes the proof since $|a_n|\le 
(|t|/(1-\beta^2))^n$, 
$|c_{n,k}|\le K_6 [n-k+1]^{\sigma_3}$, and 
$|C_n(x;\beta|q)| \le [n+1]^{\sigma_4}/(1-n
+\log_q \beta)$ and therefore
\[
\sum_{n=0}^\infty |a_n|\sum_{k=0}^{\lfloor 
n/2 \rfloor} |c_{k,n}||C_k(x;\beta|q)|\le K_7 
\sum_{n=0}^\infty \frac{|t|^n}{(1-\beta^2)^n}
(n+1)^{\sigma_5}<\infty.
\]
Therefore the theorem holds.
\end{proof}

\begin{cor} Let $x\in[-1,1]$, $|t|<1$, $\beta,\gamma\in(-1,\infty)\setminus\{0,1\}$, $|q|<1$. Then 
\begin{equation}
e^{xt}\hyp01{-}{\beta+\frac12}{\frac{(x^2-1)t^2}{4}}=
\sum_{n=0}^\infty \frac{(\beta)_nt^n}{(\gamma,2\beta)_n}C_n^\gamma(x)
\hyp23{\beta-\gamma,\beta+n}{\gamma+n+1,\beta+\frac{n}{2},\beta+\frac{n+1}{2}}{\frac{t^2}{4}}.
\label{qto1genofgenthm2}
\end{equation}
\label{cor4}
\end{cor}
\begin{proof}
In (\ref{gengenthm2}), transform $\beta\mapsto q^\beta$, $\gamma\mapsto q^\gamma$, 
$t\mapsto (1-q)t$, and take the limit as $q\uparrow 1^{-}$.
Using the definition of the $q$-exponential function \cite[(1.14.2)]{Koekoeketal} 
$E_q(z):=(-z;q)_\infty$, $\lim_{q\uparrow 1^{-}}E_q((1-q)z)=e^z$, 
and that the ${}_2\phi_1$ becomes a Kummer confluent hypergeometric functions ${}_1F_1$ 
with argument $-2it\sin\theta$. Representing this as a Bessel function of the first
kind using \cite[(10.16.5)]{NIST:DLMF}, and then using \cite[(10.2.2)]{NIST:DLMF},
the left-hand side follows. The $q\uparrow 1^{-}$ limit on the right-hand side is straightforward.
\end{proof}

\begin{thm} \label{theo:5}
Let $x\in[-1,1],$ $|t|<1-\beta^2$,  $\beta,
\gamma\in (-1,1)\setminus\{0\}$, $|q|<1$.
Then
{\small \begin{eqnarray}
\frac1
{(te^{i\theta};q)_\infty}{}_2\phi_1\!\!\left(
\!\!
\begin{array}{c}\beta,\beta e^{2i\theta}\\
\beta^2\end{array}\!\!;q,te^{-i\theta}\right)
\!=\!\sum_{n=0}^\infty\!\frac{(\beta;q)_n
t^n}{(\gamma,\beta^2;q)_n}C_n(x;\gamma|q)
\,{}_6\phi_5\!\left(\!\!
\begin{array}{c}
\beta\gamma^{-1},\beta q^n,0,0,0,0\\
\gamma q^{n+1}, \pm \beta q^{\frac n2},
\pm \beta q^{\frac {n+1}2}\end{array}
\!\!;q,\gamma t^2
\right).\label{theorem3}
\end{eqnarray}}
\end{thm}
\begin{proof}A generating function for the continuous $q$-ultraspherical/Rogers
polynomials can be found in Koekoek et al. \cite[(14.10.28)]{Koekoeketal}
\begin{equation}
\frac1{(te^{i\theta};q)_\infty}\,{}_2\phi_1
\left(
\begin{array}{c}
\beta,\beta e^{2i\theta}\\ \beta^2
\end{array};q,te^{-i\theta}\right)
=\sum_{n=0}^\infty\frac{C_n(x;\beta|q)}{(\beta^2;q)_n}t^n.
\label{genfun141028}
\end{equation}
The proof follows as above by starting with 
(\ref{genfun141028}), inserting (\ref{3:17}), 
shifting the $n$ index by $2k$, reversing
the order of summation and
using \myref{2:4} through \myref{2:12}.
\end{proof}

\begin{remark}
The $q\uparrow 1^{-1}$ limit of (\ref{theorem3}) can also be shown to be the same as 
(\ref{qto1genofgenthm2}), by using the transformation $x\mapsto -x$. 
The proof of this is the same as the proof of Corollary \ref{cor4},
except instead use the definition of the $q$-exponential function \cite[(1.14.1)]{Koekoeketal} 
$e_q(z):=1/(z;q)_\infty$, $\lim_{q\uparrow 1^{-}}e_q((1-q)z)=e^z$. Of course, the same 
is true for the $q\uparrow 1^{-}$ limits of the original generating functions 
\cite[(14.10.28--29)]{Koekoeketal}, which both are analgoues of \cite[(9.8.31)]{Koekoeketal},
and are equivalent under the transformation $x\mapsto -x$.
\end{remark}

\begin{thm} \label{theo:6}
Let $x\in[-1,1],$ $|t|<1-\beta^2$, $\gamma\in\mathbb C$, $\alpha,\beta\in (-1,1)\setminus\{0\}$,
$|q|<1$. Then
\begin{eqnarray}
\label{genFunc4C}
&&\hspace{-1cm}\frac{(\gamma te^{i\theta};q)_\infty}
{(te^{i\theta};q)_\infty}
\,{}_3\phi_2
\left(
\begin{array}{c}
\gamma,\beta,\beta e^{2i\theta}\\
\beta^2,\gamma te^{i\theta}
\end{array}
;
q,te^{-i\theta}
\right)
=
\sum_{n=0}^\infty
\frac{(\beta,\gamma;q)_n
t^n}
{(\alpha,\beta^2;q)_n}C_n(x;\alpha|q)
\nonumber\\ &&\hspace{5cm}\times\,{}_6\phi_5
\left(
\begin{array}{c}
\beta/\alpha,\beta q^n,
\pm (\gamma q^n)^{1/2},
\pm (\gamma q^{n+1})^{1/2}
\\
\alpha q^{n+1}, 
\pm \beta q^{n/2},
\pm \beta q^{(n+1)/2}
\end{array}
;
q,\alpha t^2
\right).
\end{eqnarray}
\end{thm}
\begin{proof}A generating function for 
the continuous $q$-ultraspherical/Rogers 
polynomials can be found in Koekoek et al. \cite[(14.10.33)]{Koekoeketal}
\begin{equation}
\frac{(\gamma te^{i\theta};q)_\infty}
{(te^{i\theta};q)_\infty}
\,{}_3\phi_2
\left(
\begin{array}{c}
\gamma,\beta,\beta e^{2i\theta}\\
\beta^2,\gamma te^{i\theta}
\end{array}
;
q,te^{-i\theta}
\right)
=\sum_{n=0}^{\infty}\frac{(\gamma;q)_n}
{(\beta^2;q)_n}C_n(x;\beta|q)t^n,
\label{genfun4}
\end{equation}
where $\gamma\in\mathbb C$.
Substitute \myref{3:17} into the
generating function (\ref{genfun4}),
reverse the order of summation as above, 
shift the $n$ index by $2k$, using \myref{2:4} 
through \myref{2:12}, completes
the proof.  \end{proof}
\begin{thm} \label{theo:7}
Let $x\in[-1,1],$ $|t|<\min\{(1-\beta^2) (1+\sqrt{q}|\beta|)(1-q|\gamma|),1\}$,
$\beta,\gamma\in (-1,1)\setminus\{0\},$ $|q|<1$.
Then
\begin{eqnarray}
\label{genFunthm5}
&&\hspace{-0.6cm}{}_2\phi_1\left(
\begin{array}{c}\pm \beta^{1/2}e^{i\theta}\\
-\beta\end{array}; q,{te^{-i\theta}}\right)
\,{}_2\phi_1 \left(\begin{array}{c}
\pm (q\beta)^{1/2}e^{-i\theta}
\\-q\beta\end{array}; q,{te^{i\theta}}
\right)=\sum_{n=0}^{\infty}
\frac{(\beta, \pm \beta q^{1/2};q)_nt^n}
{(\gamma,\beta^2,-q\beta;q)_n} 
C_n(x;\gamma|q)\nonumber\\
&&\hspace{-0.6cm}\times\,{}_{10}\phi_9\Bigg(
\begin{array}{c}
\beta\gamma^{-1},\beta q^n,
\pm (\beta q^{n+1/2})^{1/2},
\pm (\beta q^{n+3/2})^{1/2},
\pm i(\beta q^{n+1/2})^{1/2},
\pm i(\beta q^{n+{3/2}})^{1/2}
\\
\gamma q^{n+1},
\pm \beta q^{n/2},
\pm \beta q^{(n+1)/2},
\pm i(\beta q^{n+1})^{1/2},
\pm i(\beta q^{n+2})^{1/2}
\end{array}
; q,\gamma t^2
\Bigg).
\end{eqnarray}
\end{thm}
\begin{proof}A generating function for the 
continuous $q$-ultraspherical/Rogers polynomials 
can be found in Koekoek et al.
 \cite[(14.10.31)]{Koekoeketal}
\begin{eqnarray}
&&\hspace{-0.6cm}{}_2\phi_1
\left(
\begin{array}{c}
\pm \beta^{1/2}{e^{i\theta}}
\\
-\beta
\end{array}
;
q,{te^{-i\theta}}
\right)
\,{}_2\phi_1
\left(
\begin{array}{c}
\pm (q\beta)^{1/2}{e^{-i\theta}}
\\
-q\beta
\end{array}
;
q,{te^{i\theta}}
\right)
=\sum_{n=0}^{\infty}\frac{
(
\pm \beta q^{1/2}
;q)_n
}
{(\beta^2,-q\beta;q)_n}C_n(x;\beta|q)\,t^n.
\label{genfun2}
\end{eqnarray}
We substitute \myref{3:17} into the 
generating function (\ref{genfun2}),
switch the order of the summation, shift 
the $n$ index by $2k$ and using
\myref{2:4} through \myref{2:12},
produces 
\[
|a_n|\le \frac{|t|^n}{(1-\beta^2)^n 
(1+q|\beta|)^n(1-q|\gamma|)^n}.
\]
Therefore the theorem holds.
\end{proof}

\begin{thm} \label{theo:8}
Let $x\in[-1,1],$ $|t|<\min\{(1-\beta^2)
(1+\sqrt{q}|\beta|)(1-q|\gamma|),1\}$,
$\beta,\gamma\in (-1,1)\setminus\{0\}$, 
$|q|<1$. Then
\begin{eqnarray}
\label{genFunc1C}
&&{}_2\phi_1
\left(\begin{array}{c}
\beta^{1/2}{e^{i\theta}},(q\beta)^{1/2}
e^{i\theta}\\\beta q^{{1/2}}\end{array}
;q,{te^{-i\theta}}\right) \,{}_2\phi_1
\left(\begin{array}{c} -\beta^{1/2}
{e^{-i\theta}},-(q\beta)^{1/2}
{e^{-i\theta}}\\\beta q^{1/2}\end{array}
;q,{te^{i\theta}}\right)\nonumber\\
\nonumber &&=
\sum_{n=0}^{\infty}
\frac{(
\pm\beta,
-\beta q^{1/2};q)_n
t^n}
{(\gamma,\beta^2,\beta q^{1/2};q)_n} C_n(x;\gamma|q)\\
\nonumber
&&\times\,{}_{10}\phi_9\Bigg(
\begin{array}{c}
\beta\gamma^{-1},\beta q^n,
\pm i(\beta q^n)^{1/2},
\pm i(\beta q^{n+1})^{1/2},
\pm i(\beta q^{n+{1/2}})^{1/2},
\pm i(\beta q^{n+{3/2}})^{1/2}
\\
\gamma q^{n+1},
\pm \beta q^{n/2},
\pm \beta q^{(n+1)/2},
\pm (\beta q^{n+{1/2}})^{1/2}, 
\pm (\beta q^{n+{3/2}})^{1/2}
\end{array}
;
q,\gamma t^2
\Bigg).
\end{eqnarray}
\end{thm}
\begin{proof} We start with the generating 
function for the continuous $q$-ultraspherical/Rogers
polynomials Koekoek {\it et al.}~(2010) 
\cite[(14.10.30)]{Koekoeketal}
\begin{eqnarray}
\label{genFunc1}
{}_2\phi_1
\left(
\begin{array}{c} \beta^{1/2}e^{i\theta}
,(q\beta)^{1/2}e^{i\theta}\\ \beta q^{1/2}
\end{array};q,te^{-i\theta}\right)
\,{}_2\phi_1\left(\begin{array}{c}-\beta^{1/2}
e^{-i\theta},-(q\beta)^{1/2}e^{-i\theta}\\
\beta q^{1/2}\end{array};q,te^{i\theta}
\right)\nonumber\\
=\sum_{n=0}^{\infty}\frac{(-\beta,-\beta 
q^{1/2};q)_n}{(\beta^2,\beta
q^{1/2};q)_n} C_n(x;\beta|q)t^n.
\end{eqnarray}
Using the connection relation (\ref{3:17}) 
in (\ref{genFunc1}), reversing the orders 
of the summation, shifting the $n$ index by 
$2k$, and using \myref{2:4} through \myref{2:12},
obtains the result 
\[
|a_n|\le \frac{|t|^n}{(1-\beta^2)^n 
(1+\sqrt{q}|\beta|)^n(1-q|\gamma|)^n}.
\]
Therefore the theorem holds.
\end{proof}

\begin{thm} \label{theo:9}
Let $x\in[-1,1],$ $|t|<\min\{(1-\beta^2)
(1+\sqrt{q}|\beta|),1\}$, $\beta,\gamma
\in (-1,1)\setminus\{0\},$ $|q|<1$. Then
\begin{eqnarray}
\label{genFunc3C}
&&{}_2\phi_1\left(
\begin{array}{c}\beta^{1/2}e^{i\theta}
,-(q\beta)^{1/2}e^{i\theta}\\
-\beta q^{1/2}\end{array};
q,{te^{-i\theta}}\right)
\,{}_2\phi_1\left(\begin{array}{c}
(q\beta)^{1/2}{e^{-i\theta}}
,-\beta^{1/2}{e^{-i\theta}}\\
-\beta q^{1/2}\end{array};
q,{te^{i\theta}}\right)\nonumber\\
&&=\sum_{n=0}^{\infty}
\frac{(\pm \beta, \beta q^{1/2};q)_n
t^n}{(\gamma,\beta^2,-\beta q^{1/2};q)_n} 
C_n(x;\gamma|q)\nonumber\\
&&\times\,{}_{10}
\phi_9\Bigg(\begin{array}{c}
\beta\gamma^{-1},\beta q^n,
\pm i(\beta q^n)^{1/2},
\pm i(\beta q^{n+1})^{1/2},
\pm (\beta q^{n+1/2})^{1/2},
\pm (\beta q^{n+3/2})^{1/2}
\\
\gamma q^{n+1},
\pm \beta q^{n/2},
\pm \beta q^{(n+1)/2},
\pm i(\beta q^{n+1/2})^{1/2},
\pm i(\beta q^{n+3/2})^{1/2}
\end{array};q,\gamma t^2\Bigg).
\end{eqnarray}
\end{thm}
\begin{proof}A generating function for 
the continuous $q$-ultraspherical/Rogers 
polynomials can be found in Koekoek et al. \cite[(14.10.32)]{Koekoeketal}
\begin{eqnarray}
&&\hspace{-1cm}{}_2\phi_1
\left(\begin{array}{c}
-\beta,(\beta q)^{1/2}\\
-\beta q^{1/2}\end{array}
;q,{te^{-i\theta}}
\right)\,{}_2\phi_1
\left(\begin{array}{c}
(\beta q)^{1/2}{e^{-i\theta}},
-\beta^{1/2}{e^{-i\theta}}\\
-\beta q^{1/2}\end{array};
q,{te^{i\theta}}\right)\nonumber\\
&&\hspace{6cm}=\sum_{n=0}^{\infty}
\frac{(-\beta,\beta q^{1/2};q)_n}
{(\beta^2,-\beta q^{1/2};q)_n}
C_n(x;\beta|q)\,t^n.
\label{genfun3}
\end{eqnarray}
Similar to the proof of
\myref{genFunc1C}, we substitute 
\myref{3:17} into the generating
function
(\ref{genfun3}),
switch the order of the summation, 
shift the $n$ sum by $2k$, and use 
\myref{2:4} through \myref{2:12},
obtaining the result 
\[
|a_n|\le \frac{|t|^n}{(1-\beta^2)^n 
(1+\sqrt{q}|\beta|)^n}.
\]
Therefore the theorem holds.
\end{proof}

\section{Little $q$-Laguerre/Wall polynomials}
\label{LittleqLaguerreWallpolynomials}

The little $q$-Laguerre/Wall polynomials are defined as
\cite[(14.20.1)]{Koekoeketal}
\begin{eqnarray*}
p_n(x;a|q)\,&&:={}_2\phi_1\left(
\begin{array}{c}q^{-n},0\\ aq \end{array}
;q,qx\right)
=\frac1{(a^{-1}
q^{-n};q)_n}\, {}_2\phi_0\left(
\begin{array}{c}q^{-n},x^{-1}\\ -
\end{array}; q,\frac xa \right).
\end{eqnarray*}
The connection relation for little 
$q$-Laguerre/Wall polynomials can be
obtained by Exercise 1.33 in 
\cite{GaspRah} and using the specialization
formula which connects the little 
$q$-Laguerre/Wall polynomials with the 
little $q$-Jacobi polynomials, namely 
\cite[p. 521]{Koekoeketal}
$p_n(x;a|q)=p_n(x;a,0|q)$.
\begin{thm} \label{theo:10}
Let $a, b\in (0,q^{-1})$, $|q|<1$. Then 
the connection relation for the little 
$q$-Laguerre/Wall polynomials is given by
\begin{equation}
\label{ConCoef3}
p_n(x;a|q)=\frac{q^{-{\binom n2}}}{(q a;q)_n}
\sum_{j=0}^n \frac{q^{{\binom j2}+n(n-j)}
(-a)^{n-j}(q^{n-j+1},q b;q)_j
(b q^{1+j-n}/a;q)_{n-j}}{(q;q)_j}\,p_j(x;b |q).
\end{equation}
\end{thm}
By starting with the generating function 
for the little $q$-Laguerre/Wall polynomials
\cite[(14.20.11)]{Koekoeketal}, we derive 
generalizations using the connection relation 
for these polynomials. 
\begin{thm} \label{theo:11}
Let $(a,b)\in (0,q^{-1})$, $|q|<1$, 
$|t|<\min\{(1-q)(1-aq)/a,1\}$. Then
\[
\frac{(t;q)_\infty}{(xt;q)_\infty}
\,{}_0\phi_1 \left( \begin{array}{c}
-\\ aq \end{array};q,aqxt\right)
=\sum_{n=0}^\infty\frac{q^{\binom n2}
(-t)^n(b q;q)_n}{(q;q)_n(a q;q)_n}\,p_n(x;b|q)
\,{}_1\phi_1\left(\begin{array}{c}a/b\\ a q^{n+1}
\end{array};q,b q^{n+1} t\right).
\]
\end{thm}
\begin{proof}
We start with the generating function for 
little $q$-Laguerre/Wall polynomials found in 
Koekoek et al. \cite[(14.20.11)]{Koekoeketal}
\begin{equation}
\label{genlFunc}
\frac{(t;q)_\infty}{(xt;q)_\infty}
\,{}_0\phi_1\left(\begin{array}{c}
-\\ aq \end{array};q,aqxt \right)
=\sum_{n=0}^\infty\frac{(-1)^nq^{\binom n2}}
{(q;q)_n}p_n(x;a|q)t^n.
\end{equation}
Using the connection relation (\ref{ConCoef3}) 
in (\ref{genlFunc}), reversing the orders 
of the summations, shifting the $n$ index 
by $j$, and using \myref{2:4} 
through \myref{2:12}, obtains the desired 
result since
$|a_n|=|t|^n/(1-q)^n$, $|c_{n,k}|\le 
K_8 [n+1]^{\sigma_6}$, and
\begin{equation}
\label{lqlageq}
|p_n(x;a|q)|\le  |a|^n[n+1]^{\sigma_7}/
(1-|aq|)^n\le a^n(n+1)^{\sigma_7}/(1-aq)^n,
\end{equation}
where
$\sigma_6$ and $\sigma_7$ are independent 
of $n$ implies
\[
\sum_{n=0}^\infty |a_n|\sum_{k=0}^n |c_{k,n}| 
|p_k(x;a|q)|<\infty.
\]
Therefore the theorem holds.
\end{proof}
\section{$q$-Laguerre polynomials}
\label{qLaguerrepolynomials}

The $q$-Laguerre polynomials are defined as
\cite[(14.21.1)]{Koekoeketal}
\begin{eqnarray*}
L_n^{(\alpha)}(x;q)\,
&&=\frac{(q^{\alpha+1};q)_n}{(q;q)_n}\,
{}_1\phi_1
\left(\begin{array}{c}
q^{-n}\\q^{\alpha + 1}
\end{array}; q,-q^{n+\alpha+1}x
\right)
=\frac1{(q;q)_n}\,
{}_2\phi_1\left(
\begin{array}{c} q^{-n},-x\\0
\end{array};q,q^{n+\alpha+1}
\right).
\end{eqnarray*}
\begin{thm} \label{theo:12}
Let $\alpha,\beta\in (-1,\infty),$ $|q|<1$.
The connection relation for the $q$-Laguerre 
polynomials is given as
\begin{equation}
\label{ConCoef2}
L_n^{(\alpha)}(x;q)=
\frac{q^{n(\alpha-\beta)} }{(q;q)_n}
\sum_{j=0}^n (-1)^{n-j}q^{{\binom {n-j}2}}
(q^{n-j+1};q)_j(q^{j-n+\beta-\alpha+1};q)_{n-j}
\,L_j^{(\beta)}(x;q).
\end{equation}
\end{thm}
\begin{proof}
One could obtain the above result by following 
an analogous proof as applied  to the little 
$q$-Laguerre/Wall polynomials. Nevertheless
the result follows by using the relation 
between the little $q$-Laguerre/Wall and the 
$q$-Laguerre polynomials \cite[p.~521]{Koekoeketal}.
\end{proof}
By starting with generating functions 
for the $q$-Laguerre polynomials 
\cite[(14.21.14--16)]{Koekoeketal}, 
we derive generalizations of these
generating functions using the connection 
relation for $q$-Laguerre polynomials 
(\ref{ConCoef2}).  Note however that the 
generating function for the $q$-Laguerre 
polynomials \cite[(14.21.13)]{Koekoeketal} 
remains unchanged when one applies 
the connection relation (\ref{ConCoef2}).

\begin{thm} \label{theo:13}
Let $\alpha,\beta\in (-1,\infty),$ $|q|<1$, $|t|<(1-q^{\alpha+1})(1-q)$. Then
\begin{equation}
\frac1{(t;q)_\infty}
\,{}_0\phi_1\left(\begin{array}{c}
-\\q^{\alpha+1} \end{array};
q,-xtq^{\alpha+1}\right)=\sum_{n=0}^\infty
\frac{(q^{\alpha-\beta}t)^n
L_n^{(\beta)}(x;q)}{(q^{\alpha+1};q)_n}
\,{}_2\phi_1\left(\begin{array}{c}
q^{\alpha-\beta},0\\q^{\alpha+n+1}
\end{array};q,t\right). \label{qLag1}
\end{equation}
\end{thm}

\begin{proof}
We start with the generating function 
for $q$-Laguerre polynomials found in 
Koekoek et al. \cite[(14.21.14)]{Koekoeketal}
\begin{equation}
\label{genLFunc1}
\frac1{(t;q)_\infty}\,{}_0\phi_1
\left(\begin{array}{c}
-\\q^{\alpha+1}\end{array};
q,-xtq^{\alpha+1}\right)
=\sum_{n=0}^\infty \frac
{L_n^{(\alpha)}(x;q)}{(q^{\alpha+1};q)_n}t^n.
\end{equation}
Using the connection relation (\ref{ConCoef2}) 
in (\ref{genLFunc1}),
reversing the orders of the summations,
shifting the $n$ index by $j$, and using
\myref{2:4} through \myref{2:12},
obtains the desired result since  $|a_n|\le |t|^n/(1-q^{\alpha+1})^n$,
$|c_{n,k}|\le  K_9 [n+1]^{\beta-\alpha+4}$, and
\begin{equation}
\label{qleq}
|L_n^{(\alpha)}(x;q)|\le [n+1]^{\sigma_8}/(1-q)^n,
\end{equation}
implies
\[
\sum_{n=0}^\infty |a_n|\sum_{k=0}^{n} |c_{k,n}| |L_k^{(\alpha)}(x;q)|\le K_9 \sum_{n=0}^\infty \frac{|t|^n}{(1-q^{\alpha+1})^n(1-q)^n} 
(n+1)^{\sigma_8}<\infty.
\]
Therefore the theorem holds.
\end{proof}

\begin{thm} \label{theo:14}
Let $\alpha,\beta\in (-1,\infty),$ $|q|<1$, $|t|<(1-q^{\alpha+1})(1-q)$. Then
\begin{equation}
(t;q)_\infty\,
{}_0\phi_2\left(\begin{array}{c}
-\\
q^{\alpha+1},t\end{array}
;q,-xtq^{\alpha+1}\right)
=\sum_{n=0}^\infty\frac
{(-tq^{\alpha-\beta})^n
q^{\binom n2}L_n^{(\beta)}(x;q)}
{(q^{\alpha+1};q)_n} \,{}_1\phi_1
\left(\begin{array}{c}
q^{\alpha-\beta}\\ q^{\alpha+n+1}
\end{array};q,tq^n\right).
\label{qLag2}
\end{equation}
\end{thm}
\begin{proof}
We start with the generating function 
for the $q$-Laguerre polynomials found 
in Koekoek et al \cite[(14.21.15)]{Koekoeketal}
\begin{equation}
\label{genLFunc2}
(t;q)_\infty\, {}_0\phi_2
\left(\begin{array}{c}-\\
q^{\alpha+1},t\end{array};
q,-xtq^{\alpha+1}\right)
=\sum_{n=0}^\infty\frac
{(-t)^nq^{\binom n2}}{(q^{\alpha+1};q)_n}
L_n^{(\alpha)}(x;q).
\end{equation}
Using the connection relation 
(\ref{ConCoef2}) in (\ref{genLFunc2}),
reversing the orders of the summations,
shifting the $n$ index by $j$, and using
\myref{2:4} through \myref{2:12},
obtains the desired result since, again,
$|a_n|\le |t|^n/(1-q^{\alpha+1})^n$.
\end{proof}

\begin{thm} \label{theo:15}
Let $\alpha,\beta\in (-1,\infty)$, 
$\gamma\in\mathbb C$, $|q|<1$, 
$|t|<1-q$. Then
\begin{equation}
\frac{(\gamma t;q)_\infty}{(t;q)_\infty}
\,{}_1\phi_2\left(\begin{array}{c}
\gamma\\ q^{\alpha+1},\gamma t
\end{array};q,-xtq^{\alpha+1}
\right)=\sum_{n=0}^\infty
\frac{(\gamma;q)_n(tq^{\alpha-\beta})^n}
{(q^{\alpha+1};q)_n}L_n^{(\beta)}(x;q)
\,{}_2\phi_1\left(\begin{array}{c}
q^{\alpha-\beta},\gamma q^n\\
q^{\alpha+n+1}\end{array}
;q,t\right).\label{qLag3}
\end{equation}
\end{thm}
\begin{proof}
We start with the generating function for the 
$q$-Laguerre polynomials found in Koekoek et al. \cite[(14.21.15)]{Koekoeketal}
\begin{equation}
\label{genLFunc3}
\frac{(\gamma t;q)_\infty}{(t;q)_\infty}
\,{}_1\phi_2
\left(
\begin{array}{c}
\gamma\\
q^{\alpha+1},\gamma t
\end{array}
;q,-xtq^{\alpha+1}
\right)
=\sum_{n=0}^\infty
\frac{(\gamma;q)_n}{(q^{\alpha+1};q)_n}
L_n^{(\alpha)}(x;q)t^n.
\end{equation}
Using the connection relation 
(\ref{ConCoef2}) in (\ref{genLFunc3}),
reversing the orders of the summations,
shifting the $n$ index by $j$, and using
\myref{2:4} through \myref{2:12},
obtains the result
\[
|a_n|\le \frac{|t|^n [n+1]^{|\gamma|+1}}
{\alpha+1}.
\]
Therefore the theorem holds.
\end{proof}

\section{Askey-Wilson polynomials}
\label{AskeyWilsonpolynomials}

The Askey-Wilson polynomials are defined as
\cite[(14.1.1)]{Koekoeketal}
\[
p_n(x;{\bf a}|q):=a^{-n}(ab,ac,ad;q)_n\,
\qhyp43{q^{-n},abcdq^{n-1},ae^{i\theta},
ae^{-i\theta}}{ab,ac,ad}{q,q},
\]
where ${\bf a}:=\{a,b,c,d\}$.
Throughout this section and the next section 
on the continuous $q$-ultraspherical/Rogers
polynomials, $x=\cos\theta$.
We derive generalizations of generating 
functions for the Askey-Wilson polynomials 
\cite[(14.1.13--15)]{Koekoeketal} using its 
connection relation with one free parameter 
\cite[(7.6.8--9)]{GaspRah}
\begin{equation}
p_n(x;{\bf a}|q)=\sum_{k=0}^nc_{k,n}({\bf a}
;\alpha|q)p_k(\alpha,b,c,d|q),
\label{conncoefAW}
\end{equation}
where
\[
c_{k,n}:=\frac{\alpha^{n-k}(a/\alpha;q)_{n-k}
(abcdq^{n-1};q)_k(q,bc,bd,cd;q)_n}{(q,bc,bd,cd,
\alpha bcdq^{k-1};q)_k(q,\alpha bcdq^{2k};q)_{n-k}}.
\]
Due to the symmetry in $a,b,c,d$ of the 
Askey-Wilson polynomials, the generating 
functions \cite[(14.1.13--15)]{Koekoeketal} 
are all equivalent. Therefore, we will 
only consider \cite[(14.1.13)]{Koekoeketal}.
\begin{thm} \label{theo:2}
Let $x\in[-1,1],$ $|t|<(1-q)^3$, $a,b,c,d
\in\mathbb R$ or occur in complex conjugate
pairs if complex, and 
$\max(|a|,|b|,|c|,|d|)<1$, $|q|<1$. Then
\begin{eqnarray}
\hspace{-0.3cm}
&&\qhyp21{ae^{i\theta},be^{i\theta}}
{ab}{q,te^{-i\theta}}
\qhyp21{ce^{-i\theta},de^{-i\theta}}
{cd}{q,te^{i\theta}}\nonumber
=\sum_{n=0}^\infty
\frac{t^n(\alpha bcd/q, \pm \sqrt{abcd/q},
\pm \sqrt{abcd};q)_n}{(ab,cd,abcd/q,
\sqrt{\alpha bcd/q}, \pm \sqrt{\alpha bcd}
;q)_n} \nonumber\\ &&\hspace{4cm}\times
\qhyp43{a/\alpha,bcq^n,bdq^n,abcdq^{2n-1}}
{abq^n, abcdq^{n-1}, \alpha bcdq^{2n}}
{q,\alpha t}\frac{p_n(x;\alpha,b,c,d|q)}
{(q;q)_n}. \label{awthm}
\end{eqnarray}
\end{thm}
\begin{proof}A generating function for 
the Askey-Wilson polynomials can be found 
in Koekoek et al. \cite[(14.1.13)]{Koekoeketal}
\begin{equation}
\qhyp21{ae^{i\theta},be^{i\theta}}{ab}
{q,te^{-i\theta}}\qhyp21{ce^{-i\theta},de^{-i\theta}}
{cd}{q,te^{i\theta}}=\sum_{n=0}^\infty
\frac{t^n p_n(x;{\bf a}|q)}{(q,ab,cd;q)_n}.
\label{genfun14113}
\end{equation}
Using the connection relation for these 
polynomials in this generating function
produces a double infinite sum. In order 
to justify reversing the summation symbols
we show that
\[
\sum_{n=0}^\infty |a_n|\sum_{k=0}^n 
|c_{k,n}| |p_n(x;\alpha,b,c,d|q)|<\infty.
\]
Taking into account Lemma \ref{lem:4} we have
\[
|a_n|\le K_1 \frac{|t|^n[n+1]^2_q}{(q;q)_n^3},
\]
with $K_1=\max\{1,(\Re(A+\beta)
\Re(\gamma+\delta))^{-1}\}$, being 
$a=q^A$, $b=q^\beta$, $c=q^\gamma$, 
and $d=q^\delta$,
\begin{equation}
\label{pneq}
|p_n(\alpha,b,c,d|q)|\le K_2 
\sum_{k=0}^n [k+1]_q^\sigma\le 
K_2 (n+1)^{\sigma+1},
\end{equation}
with $K_2=\max\{1,(\Re(A+\beta)\Re(A+\gamma)
\Re(A+\delta)^{-1}\}$, and $\sigma=3A+\beta
+\gamma+\delta+3$. Next, following an 
analogous idea we can  find as in the proof 
of \cite[Theorem 1]{CohlMacK} that
\[
|c_{k,n}|\le K_4 (1+k)^{\sigma_2+1}(1+n)^{\sigma_2},
\]
where $K_4$ and $\sigma_2$ are constants 
independent of $n$. Combining these results 
demonstrates
\[
\sum_{n=0}^\infty |a_n|\sum_{k=0}^n |c_{k,n}| 
|p_k(x;\alpha,b,c,d|q)|\le K_5 \sum_{n=0}^\infty 
\frac {|t|^n}{(1-q)^{3n}} (n+1)^{3\sigma+5}<\infty.
\]
Hence, the proof follows as above by starting 
with (\ref{genfun14113}), inserting 
(\ref{conncoefAW}), shifting the $n$ index 
by $k,$ reversing the order of summation 
and using \myref{2:1}--\myref{2:7},
\myref{2:9}, \myref{2:12}.
\end{proof}

\section{Definite integrals, infinite series, and $q$-integrals}
\label{Definiteintegralsinfiniteseriesandqintegrals}

Consider a sequence of orthogonal polynomials 
$(p_k(x;{\boldsymbol \alpha}))$ (over a domain 
$A$, with positive weight $w(x;{\boldsymbol 
\alpha})$) associated with a linear functional 
${\bf u}$, where ${\boldsymbol\alpha}$ is a 
set of fixed parameters.
Define $s_k,$ $k\in\mathbb N_0$ by
\[
s^2_k:=\int_A \big\{p_k(x;{\boldsymbol 
\alpha})\big\}^2\, w(x;{\boldsymbol \alpha})\, dx.
\]
In order to justify interchange between a 
generalized generating function via connection 
relation and an orthogonality relation for $p_k$,
we show that the double sum/integral 
converges in the $L^2$-sense with respect
to the weight $w(x;{\boldsymbol \alpha})$.
This requires
\begin{eqnarray} \label{form-inve-L2}
\sum_{k=0}^\infty d_k^2 s_k^2<\infty,
\end{eqnarray}
where
\[
d_k=\sum_{n=k}^\infty a_n c_{k,n}.
\]
Here, $a_n$ is the coefficient multiplying 
the orthogonal polynomial in the
original generating function, and $c_{k,n}$ 
is the connection coefficient for $p_k$
(with appropriate set of parameters).

\begin{lemma}
\label{lemmasum}
Let $\bf u$ be a classical linear functional 
and let $(p_n(x))$, $n\in\mathbb N_0$ be 
the sequence of orthogonal polynomials 
associated with $\bf u$.
If $|p_n(x)|\le K(n+1)^\sigma \gamma^n$, 
with $K$, $\sigma$ and $\gamma$ constants 
independent of $n$, then
$|s_n|\le  K(n+1)^\sigma \gamma^n |s_0|$.
\end{lemma}
\begin{proof}
Let $n\in\mathbb N_0$, then
\[
s_n^2=\langle {\bf u}, p^2_n\rangle \le 
\left(K(n+1)^\sigma \gamma^n\right)^2
\langle {\bf u}, 1\rangle= \left(K(n+1)^\sigma 
\gamma^n\right)^2  s_0^2.
\]
The result follows.
\end{proof}

Given
$|p_k(x;{\boldsymbol \alpha})|\le 
K(k+1)^\sigma \gamma^k$, with $K$, $\sigma$ 
and $\gamma$ constants independent of $k$,
an orthogonality relation for $p_k$, and 
$|t|<1/\gamma$, one has
\[
\sum_{n=0}^\infty|a_n|\sum_{k=0}^n|c_{k,n}s_k|<\infty,
\]
which implies
\[
\sum_{k=0}^\infty |d_ks_k|<\infty.
\]
Therefore one has confirmed (\ref{form-inve-L2}),
indicating that we are justified in 
reversing the order of our generalized sums
and the orthogonality relations under 
the above assumptions.

All polynomial families used throughout 
this paper fulfill such assumptions. See 
for instance (\ref{pneq}), (\ref{lqlageq}), 
(\ref{qleq}). Such inequalities depend 
entirely on the representation of the 
linear functional. In this section one 
has integral representations, infinite 
series, and representations in terms of 
the $q$-integral. In all the cases
Lemma \ref{lemmasum} can be applied and 
we are justified in interchanging 
the linear form and the infinite sum.
\medskip

\subsection{Definite integrals}

\subsubsection{Continuous $q$-ultraspherical/Rogers polynomials}

The property of orthogonality for continuous $q$-ultraspherical/Rogers
polynomials found in Koekoek {\it et al.} (2010) \cite[(3.10.16)]{Koekoeketal} is given by
\begin{equation}
\label{ortho}
\int_{-1}^1 C_m(x;\beta|q)C_n(x;\beta|q)
\frac{w(x;\beta)}{\sqrt{1-x^2}}\, dx=
2\pi\frac{(1-\beta)(\beta,q\beta;q)_\infty 
(\beta^2;q)_n}{(1-\beta q^n)
(\beta^2,q;q)_\infty (q;q)_n}\delta_{mn},
\end{equation}
where $w:(-1,1)\to[0,\infty)$ is the weight 
function defined by
\begin{equation}
\label{weight}
w(x;\beta|q):=
\left|
\frac{(e^{2i\theta};q)_\infty}
{(\beta e^{2i\theta};q)_\infty}
\right|
^2.
\end{equation}
We will use this orthogonality relation 
for proofs of the following definite
integrals.

\begin{cor}
Let $n\in\mathbb N_0,$ $\beta,\gamma
\in (-1,1)\setminus\{0\}$, $|q|<1$, 
$|t|<1-\beta^2$. Then
\begin{eqnarray*}
\int^1_{-1} && (te^{-i\theta};q)_\infty
\,{}_2\phi_1
\left(
\begin{array}{c}\beta , \beta 
e^{2i\theta}\\ \beta^2\end{array}
;q,te^{-i\theta}\right)
C_n(x;\gamma|q)\frac{w(x;\gamma)}
{\sqrt{1-x^2}} dx\\&&
=2\pi(-\beta t)^n\frac{q^{\binom n 2}
(\gamma,q\gamma;q)_\infty
(\beta,\gamma^2;q)_n}{(\gamma^2,q;q)_\infty
(q,\beta^2,q\gamma;q)_n}
{}_2\phi_5\left(\begin{array}{c}
\beta\gamma^{-1},\beta q^n\\
\gamma q^{n+1}, \pm \beta q^{n/2},
\pm \beta q^{(n+1)/2}
\end{array};q,\gamma(\beta t)^2 
q^{2n+1}\right). \nonumber
\end{eqnarray*}
\end{cor}
\begin{proof}Using the generalized generating 
function (\ref{gengenthm2}) and (\ref{ortho}),
the proof follows as above.
\end{proof}

\begin{cor}
Let $n\in\mathbb N_0,$ $\beta,\gamma\in 
(-1,1)\setminus\{0\}$, $|q|<1$, 
$|t|<1-\beta^2$. Then
\begin{eqnarray*}
\label{int2}
\int^1_{-1} 
\frac1{(te^{i\theta};q)_\infty}
&& \,{}_2\phi_1\left(\begin{array}{c}
\beta , \beta e^{2i\theta}\\\beta^2
\end{array};q,te^{-i\theta}
\right)C_n(x;\gamma|q)\frac{w(x;\gamma)}
{\sqrt{1-x^2}} dx\\ &&
=2\pi t^n \frac{(\gamma,q\gamma;q)_\infty
(\beta,\gamma^2;q)_n}{(\gamma^2,q;q)_\infty
(q,\beta^2,q\gamma;q)_n}
{}_6\phi_5\left(\begin{array}{c}
\beta\gamma^{-1},\beta q^n,0,0,0,0\\
\gamma q^{n+1}, \pm \beta q^{n/2},
\pm \beta q^{(n+1)/2}\end{array}
;q,\gamma t^2\right).\nonumber
\end{eqnarray*}
\end{cor}
\begin{proof}
We complete the proof using (\ref{theorem3}) 
and \myref{ortho}.
\end{proof}

\begin{cor}
Let $n\in\mathbb N_0,$ $\gamma\in\mathbb C$, $\alpha,\beta\in (-1,1)\setminus\{0\},$ $|q|<1$,
$|t|<1-\beta^2$. Then
\begin{eqnarray*}
\int^1_{-1}  &&
\frac{(\gamma te^{i\theta};q)_\infty}
{(te^{i\theta};q)_\infty}
 \,{}_3\phi_2 \left(
\begin{array}{c}
\gamma,\beta,\beta e^{2i\theta}\\
\beta^2,\gamma te^{i\theta}
\end{array};q,te^{-i\theta}
\right)
C_n(x;\alpha|q)\frac{w(x;\alpha)}{\sqrt{1-x^2}} dx\\
&&\hspace{0.5cm}=2\pi t^n
\frac{(\alpha,q\alpha;q)_\infty(\alpha^2,\gamma,\beta;q)_n}
{(\alpha^2,q;q)_\infty(q,\beta^2,q\alpha;q)_n}
{}_6\phi_5\left(
\begin{array}{c}
\beta/\alpha,\beta q^n,
\pm \left(\gamma q^n\right)^{1/2}
\pm \left(\gamma q^{n+1}\right)^{1/2}
\\
\alpha q^{n+1},
\pm \beta q^{n/2},
\pm \beta q^{(n+1)/2}
\end{array}
;q,\alpha t^2
\right).
\nonumber
\end{eqnarray*}
\end{cor}
\begin{proof}
We complete the proof using
(\ref{genFunc4C})
and \myref{ortho}.
\end{proof}

\begin{cor} \label{cor:22}
Let $n\in\mathbb N_0,$ $\beta,\gamma\in (-1,1)\setminus\{0\},$ $|q|<1$,
$|t|<\min\{(1-\beta^2) (1+\sqrt{q}|\beta|)(1-q|\gamma|),1\}$.
Then
\begin{eqnarray*}
&&\hspace{0.0cm}\int^1_{-1}{}_2\phi_1
\left(
\begin{array}{c}
\pm \beta^{1/2}e^{i\theta}
\\
-\beta
\end{array}
;
q,te^{-i\theta}
\right)
\,{}_2\phi_1
\left(
\begin{array}{c}
\pm (q\beta)^{1/2}e^{-i\theta}
\\
-q\beta
\end{array}
;
q,te^{i\theta}
\right)
C_n(x;\gamma|q)\frac{w(x;\gamma)}{\sqrt{1-x^2}} dx\\
&&=2\pi t^n
\frac{(\gamma,q\gamma;q)_\infty(\gamma^2,\beta,
\pm \beta q^{1/2}
;q)_n}{(\gamma^2,q;q)_\infty(\beta^2,-q\beta,q\gamma,q;q)_n}\\[0.2cm]
&&\times\,{}_{10}\phi_9\Bigg(
\begin{array}{c}
\beta\gamma^{-1},\beta q^n,
\pm (\beta q^{n+1/2})^{1/2},
\pm (\beta q^{n+3/2})^{1/2},
\pm i(\beta q^{n+1/2})^{1/2},
\pm i(\beta q^{n+{3/2}})^{1/2}
\\
\gamma q^{n+1},
\pm \beta q^{n/2},
\pm \beta q^{(n+1)/2},
\pm i(\beta q^{n+1})^{1/2},
\pm i(\beta q^{n+2})^{1/2}
\end{array}
; q,\gamma t^2 \Bigg).
\end{eqnarray*}
\end{cor}
\begin{proof}
We complete the proof using (\ref{genFunthm5}) and \myref{ortho}.
\end{proof}

\begin{cor} \label{cor:23}
Let $n\in\mathbb N_0,$ $\beta,\gamma\in (-1,1)\setminus\{0\},$ $|q|<1$,
$|t|<\min\{(1-\beta^2)(1+\sqrt{q}|\beta|)(1-q|\gamma|),1\}$.
Then
\begin{eqnarray*}
\int^1_{-1}&&{}_2\phi_1
\left(
\begin{array}{c} \beta^{1/2}e^{i\theta},
(q\beta)^{1/2}e^{i\theta}\\ \beta q^{1/2}
\end{array};q,te^{-i\theta}\right)
\,{}_2\phi_1\left(\begin{array}{c}
-\beta^{1/2}e^{-i\theta},-(q\beta)^{1/2}
e^{-i\theta}\\\beta q^{1/2}\end{array};
q,te^{i\theta}\right) \\
&&\times 
C_n(x;\gamma|q)\frac{w(x;\gamma)}
{\sqrt{1-x^2}} dx=2\pi t^n \frac
{(\gamma,q\gamma;q)_\infty(\gamma^2,
\pm\beta,-\beta q^{1/2};q)_n}
{(\gamma^2,q;q)_\infty(\beta^2,\beta q^{1/2},q\gamma,q;q)_n}\\[0.2cm]
&&\times\,{}_{10}\phi_9\Bigg(
\begin{array}{c}
\beta\gamma^{-1},\beta q^n,
\pm i(\beta q^n)^{1/2},
\pm i(\beta q^{n+1})^{1/2},
\pm i(\beta q^{n+1/2})^{1/2}
\pm i(\beta q^{n+3/2})^{1/2}
\\
\gamma q^{n+1},
\pm \beta q^{n/2},
\pm \beta q^{(n+1)/2},
\pm (\beta q^{n+1/2})^{1/2},
\pm (\beta q^{n+3/2})^{1/2}
\end{array}
;
q,\gamma t^2
\Bigg).
\end{eqnarray*}
\end{cor}
\begin{proof}
We complete the proof using (\ref{genFunc1C}) and \myref{ortho}.
\end{proof}

\begin{cor} \label{cor:24}
Let $n\in\mathbb N_0$ $\beta,\gamma\in (-1,1)\setminus\{0\},$ $|q|<1$, 
$|t|<\min\{(1-\beta^2)(1+\sqrt{q}|\beta|),1\}$.
Then
\begin{eqnarray*}
\int^1_{-1}&&{}_2\phi_1
\left(\begin{array}{c} \beta^{1/2}
e^{i\theta},-(q\beta)^{1/2}e^{i\theta}\\
-\beta q^{1/2}\end{array};q,te^{-i\theta}
\right)
\,{}_2\phi_1\left(\begin{array}{c}
(q\beta)^{1/2}e^{-i\theta},-\beta^{1/2}
e^{-i\theta}\\-\beta q^{1/2}\end{array};
q,te^{i\theta}\right)
\\&&
\times C_n(x;\gamma|q)\frac{w(x;\gamma)}
{\sqrt{1-x^2}} dx=2\pi t^n\frac{(\gamma,
q\gamma;q)_\infty(\gamma^2,\pm \beta
\beta q^{1/2};q)_n}{(\gamma^2,q;q)_\infty
(\beta^2,-\beta q^{1/2},q\gamma,q;q)_n}
\\[0.2cm] &&\times\,{}_{10}\phi_9\Bigg(
\begin{array}{c}
\beta\gamma^{-1},\beta q^n,
\pm i(\beta q^n)^{1/2},
\pm i(\beta q^{n+1})^{1/2},
\pm (\beta q^{n+1/2})^{1/2},
\pm (\beta q^{n+3/2})^{1/2},\\
\gamma q^{n+1},
\pm \beta q^{n/2},
\pm \beta q^{(n+1)/2},
\pm i(\beta q^{n+1/2})^{1/2},
\pm i(\beta q^{n+3/2})^{1/2},
\end{array};q,\gamma t^2\Bigg).
\end{eqnarray*}
\end{cor}
\begin{proof}
We complete the proof using 
(\ref{genFunc3C}) and \myref{ortho}.
\end{proof}

\subsubsection{$q$-Laguerre polynomials}

The continuous orthogonality relation 
for $q$-Laguerre
polynomials is given by the following 
result.  Notice that this result appears 
in \cite[Section 2]{Christiansen}.
\begin{prop} \label{theo:25}
Let $\alpha\in(-1,\infty)$, 
$m,n\in\mathbb N_0$, $|q|<1$. 
Then
\begin{eqnarray}
&& \hspace{-0.5cm}\int_0^\infty 
L_m^{(\alpha)}(x;q) L_n^{(\alpha)}(x;q)
\frac{x^\alpha}{(-x;q)_\infty} dx
=-\frac{\delta_{m,n}}{q^n}
\left\{ \begin{array}{l@{\ \mathrm{if}\, }l}
\displaystyle \frac{\pi(q^{-\alpha}
;q)_\infty(q^{\alpha+1};q)_n}
{\sin(\pi\alpha)(q;q)_\infty(q;q)_n},
&\alpha\in(-1,\infty)\setminus\mathbb N_0, 
\\[0.6cm] \displaystyle \frac
{(q^{n+1};q)_\alpha\log q}
{q^{\alpha(\alpha+1)/2}},
&\alpha\in\mathbb N_0.
\end{array}\right.
\label{HSC1}
\end{eqnarray}
\end{prop}

\begin{proof}
The continuous orthogonality relation 
for the $q$-Laguerre polynomials is 
given in \cite[(14.21.2)]{Koekoeketal} 
with the right-hand side expressed in 
terms of gamma functions, namely
\[
\int_0^\infty L_m^{(\alpha)}(x;q) 
L_n^{(\alpha)}(x;q) \frac{x^\alpha}
{(-x;q)_\infty} dx=\frac{(q^{-\alpha}
;q)_\infty (q^{\alpha+1};q)_n}
{q^n(q;q)_\infty(q;q)_n}
\Gamma(-\alpha)\Gamma(\alpha+1)
\delta_{m,n}.
\]
The gamma functions can be replaced using 
the reflection formula \cite[(5.5.3)]{NIST:DLMF} 
and the result is given in the theorem for
$\alpha\in(-1,\infty)\setminus\mathbb N_0$.
The result for $\alpha\in\mathbb N_0$ is 
a consequence of (\ref{2:3}) and 
\cite[cf. (2.9)]{Askey80}, namely
\[
\lim_{\alpha\to k} \frac{(q^{1-\alpha}
;q)_\infty}{\sin(\pi\alpha)(aq^{-\alpha}
;q)_\infty}
=\frac{-(q;q)_\infty (q;q)_{k-1} \log q}
{\pi q^{\binom k2} (a;q)_\infty (a;q^{-1})_k},
\]
which leads to
\[
\lim_{\alpha\to k} \frac
{(q^{-\alpha};q)_\infty}
{\sin(\pi\alpha)}
=\frac{(q;q)_\infty (q;q)_k \log q}
{\pi q^{k(k+1)/2}}.
\]
Applying this limit completes the proof.
\end{proof}

\begin{cor} \label{cor:26}
Let $n\in\mathbb N_0$, $\alpha,\beta\in 
(-1,\infty),$ $|q|<1$,  $|t|<(1-q^{\alpha
+1})(1-q)$. Then
\begin{eqnarray*}
&&\int_0^\infty \qhyp01-{q^{\alpha+1}}{q,-xtq^{\alpha+1}}L_n^{(\beta)}(x;q)
\frac{x^{\beta}}{(-x;q)_\infty}dx
\\[0.2cm]
&&\hspace{1.3cm}
=\frac{-\left(tq^{\alpha-\beta}\right)^n(t;q)_\infty}
{q^n(q^{\alpha+1};q)_n}
\qhyp21{q^{\alpha-\beta},0}{q^{\alpha+n+1}}{q,t}
\left\{ \begin{array}{l@{\, \mathrm{if}\,}l}
\displaystyle \frac{\pi(q^{-\beta};q)_\infty(q^{\beta+1};q)_n}
{\sin(\pi\beta)(q;q)_\infty(q;q)_n},
&  \beta\in(-1,\infty)\setminus\mathbb N_0, \\[0.6cm]
\displaystyle \frac{(q^{n+1};q)_\beta\log q}{q^{\beta(\beta+1)/2}},
&  \beta\in\mathbb N_0.
\end{array}\right.
\end{eqnarray*}
\end{cor}
\begin{proof}
Using (\ref{HSC1}) with (\ref{qLag1})
completes the proof.
\end{proof}

\begin{cor}\label{cor:27}
Let $n\in\mathbb N_0,$ $\alpha,\beta\in (-1,\infty),$ $|q|<1$, 
$|t|<(1-q^{\alpha+1})(1-q)$. Then
\begin{eqnarray*}
&&\hspace{-0.2cm}\int_0^\infty \qhyp02-{q^{\alpha+1},t}{q,-xtq^{\alpha+1}}
L_n^{(\beta)}(x;q)\frac{x^{\beta}}{(-x;q)_\infty}dx \\[0.2cm]
&&\hspace{0.4cm}
=\frac{-\left(-tq^{\alpha-\beta}\right)^n}
{q^n(t;q)_\infty(q^{\alpha+1};q)_n}
\qhyp11{q^{\alpha-\beta}}{q^{\alpha+n+1}}{q,tq^n}
\left\{ \begin{array}{l@{\ \mathrm{if}\,}l}
\displaystyle \frac{\pi(q^{-\beta};q)_\infty(q^{\beta+1};q)_n}
{\sin(\pi\beta)(q;q)_\infty(q;q)_n},
& \beta\in(-1,\infty)\setminus\mathbb N_0, \\[0.6cm]
\displaystyle \frac{(q^{n+1};q)_\beta\log q}{q^{\beta(\beta+1)/2}},
& \beta\in\mathbb N_0.
\end{array}\right.
\end{eqnarray*}
\end{cor}
\begin{proof}
Using (\ref{HSC1}) with (\ref{qLag2})
completes the proof.
\end{proof}

\begin{cor}\label{cor:28}
Let $n\in\mathbb N_0,$ $\alpha,\beta\in (-1,\infty),$ $\gamma\in\mathbb C$, 
$|q|<1$, $|t|<1-q$.
Then
\begin{eqnarray*}
&&\int_0^\infty \qhyp12{\gamma}{q^{\alpha+1},\gamma t}{q,-xtq^{\alpha+1}}
L_n^{(\beta)}(x;q)\frac{x^{\beta}}{(-x;q)_\infty}dx
=\frac{-\left(tq^{\alpha-\beta}\right)^n
(t;q)_\infty(\gamma;q)_n}{q^n(\gamma t
;q)_\infty(q^{\alpha+1};q)_n}
\\[0.2cm] &&\hspace{4.3cm}\times
\qhyp21{q^{\alpha-\beta},\gamma q^n}{q^{\alpha+n+1}}{q,t}
\left\{ \begin{array}{l@{\, \mathrm{if}\,}l}
\displaystyle \frac{\pi(q^{-\beta};q)_\infty(q^{\beta+1};q)_n}
{\sin(\pi\beta)(q;q)_\infty(q;q)_n},
&  \beta\in(-1,\infty)\setminus\mathbb N_0, \\[0.6cm]
\displaystyle \frac{(q^{n+1};q)_\beta\log q}{q^{\beta(\beta+1)/2}},
& \beta\in\mathbb N_0.
\end{array}\right.
\end{eqnarray*}
\end{cor}
\begin{proof}
Using (\ref{HSC1}) with (\ref{qLag3})
completes the proof.
\end{proof}

\subsubsection{Askey-Wilson polynomials}
 
The orthogonality relation for the Askey-Wilson polynomials is given
by \cite[(14.1.2)]{Koekoeketal}
\begin{equation}
\label{orthoAW}
\int_{-1}^1 p_m(x;{\bf a}|q) p_n(x;{\bf a}|q)\frac{w(x;{\bf a};q)}{\sqrt{1-x^2}}dx
=2\pi h_n({\bf a};q) \delta_{m,n},
\end{equation}
where ${\bf a}:=\{a,b,c,d\}$, $w:[-1,1]\to[0,\infty)$ is defined by
\begin{equation}
w(x;{\bf a};q):=\left|
\frac{(e^{2i\theta};q)_\infty}
{(ae^{i\theta},be^{i\theta},ce^{i\theta},de^{i\theta};q)_\infty}
\right|^2,
\label{weightAW}
\end{equation}
and
\[
h_n({\bf a};q):=
\frac{(abcdq^{n-1};q)_n(abcdq^{2n};q)_\infty}
{(q^{n+1},abq^n,acq^n,adq^n,bcq^n,bdq^n,cdq^n;q)_\infty}.
\]
\begin{cor}
Let $n\in\mathbb N_0,$ $\alpha,a,b,c,d\in\mathbb R$ or occur in complex conjugate
pairs if complex, and $\max(|\alpha|,|a|,|b|,|c|,|d|)<1,$ $|q|<1$, $|t|<(1-q)^3$. Then
\begin{eqnarray*}
\hspace{-0.3cm}
&&\int_{-1}^1\qhyp21{ae^{i\theta},be^{i\theta}}{ab}{q,te^{-i\theta}}
\qhyp21{ce^{-i\theta},de^{-i\theta}}{cd}{q,te^{i\theta}}
p_n(x;\alpha,b,c,d|q)\frac{w(x;\alpha,b,c,d;q)}{\sqrt{1-x^2}}dx \nonumber\\
&&\hspace{2cm}=
\frac{2\pi t^n(\alpha bcdq^{2n};q)_\infty
(\pm \sqrt{abcd/q},
\pm \sqrt{abcd},
;q)_n}
{(q^{n+1},\alpha bq^n,\alpha cq^n,\alpha dq^n,bcq^n,bdq^n,cdq^n;q)_\infty
(q,ab,cd,abcd/q;q)_n}\nonumber\\
&&\hspace{4cm}\times\qhyp43{a/\alpha,bcq^n,bdq^n,abcdq^{2n-1}}
{abq^n,abcdq^{n-1},\alpha bcdq^{2n}}{q,\alpha t}.
\end{eqnarray*}
\end{cor}
\begin{proof}We begin with the generalized generating function
\myref{awthm},
multiply both sides by
\[
p_n(x;\alpha,b,c,d|q)\frac{w(x;\alpha,b,c,d;q)}
{\sqrt{1-x^2}},
\]
where $w(x;\alpha,b,c,d;q)$ is obtained from \myref{weightAW},
and integrate over $(-1,1)$ using the orthogonality relation (\ref{orthoAW}),
producing the desired result.
\end{proof}

\subsection{Infinite series}

\subsubsection{Little $q$-Laguerre/Wall polynomials}

The little $q$-Laguerre/Wall polynomials satisfy
a discrete orthogonality relation, namely
\cite[(14.20.2)]{Koekoeketal}
\[
\sum_{k=0}^\infty
p_m(q^k;a|q)
p_n(q^k;a|q)
\frac{(aq)^k}{(q;q)_k}
=\frac{(aq)^n(q;q)_n}{(aq;q)_\infty(aq;q)_n}
\delta_{m,n},
\]
for $a\in(0,1/q)$, with $|q|<1$.

\begin{cor} \label{cor:29}
Let $n\in\mathbb N_0,$ $|q|<1$, $\alpha,\beta\in(0,q^{-1})$, 
$|t|<\min\{(1-\beta^2)(1+\sqrt{q}|\beta|),1\}$.
Then
\begin{eqnarray*}
&& \sum_{k=0}^\infty \frac{(q\beta)^k}
{(tq^k;q)_\infty}
\qhyp01{-}{q\alpha}{q,t\alpha q^{k+1}}
\dfrac{p_n\left(q^k;\beta|q\right)}{(q;q)_k}=
\frac{q^{\binom n 2}(-q\beta t)^n}{(t,q\beta;q)_\infty(q\alpha;q)_n}\qhyp11
{\alpha/\beta}{\alpha q^{n+1}}{q,t\beta q^{n+1}}.
\end{eqnarray*}
\end{cor}
\begin{proof}
We begin with the generalized generating function (\ref{genFunc3C}) and
using \myref{ortho} completes the proof.  This orthogonality isn't there.
\end{proof}

\subsubsection{$q$-Laguerre polynomials}

One type of discrete orthogonality that the $q$-Laguerre polynomials satisfy is
\cite[(14.21.3)]{Koekoeketal}
\begin{equation}
\sum_{k=-\infty}^\infty L_m^{(\alpha)}(cq^k;q)L_n^{(\alpha)}(cq^k;q)
\frac{q^{(\alpha+1)k}}{(-cq^k;q)_\infty}
=\frac{(q,-cq^{\alpha+1},-q^{-\alpha}/c;q)_\infty(q^{\alpha+1};q)_n}
{q^n(q^{\alpha+1},-c,-q/c;q)_\infty(q;q)_n}\delta_{m,n},
\label{bilateralorthogqLag}
\end{equation}
for $\alpha\in(-1,\infty)$, $c>0$.

\begin{cor} \label{cor:30}
Let $n\in\mathbb N_0,$ $|q|<1$, $\alpha,\beta\in\left(-1,\infty\right)$, 
$|t|<(1-q^{\alpha+1})(1-q)$, $c>0$.
 Then
\begin{eqnarray*}
&&\sum_{k=-\infty}^\infty
\qhyp01-{q^{\alpha+1}}{q,-ctq^{\alpha+k+1}}L_n^{(\beta)}(cq^k;q)
\frac{q^{(\beta+1)k}}{(-cq^k;q)_\infty}\nonumber\\[0.2cm]
&&\hspace{3cm}=\frac{\left(tq^{\alpha-\beta}\right)^n(t,q,-cq^{\beta
+1},-q^{-\beta}/c;q)_\infty(q^{\beta+1};q)_n}
{q^n(q^{\beta+1},-c,-q/c;q)_\infty(q,q^{\alpha+1};q)_n}
\qhyp21{q^{\alpha-\beta},0}{q^{\alpha+n+1}}{q,t}.
\end{eqnarray*}
\begin{proof}
This follows using
(\ref{qLag1})
with (\ref{bilateralorthogqLag}).
\end{proof}
\end{cor}

\begin{cor}\label{cor:31}
Let $n\in\mathbb N_0,$ $|q|<1$, $\alpha,\beta\in\left(-1,\infty\right)$, 
$|t|<(1-q^{\alpha+1})(1-q)$,
$c>0$. Then
\begin{eqnarray*}
&&\sum_{k=-\infty}^\infty
\qhyp02-{q^{\alpha+1},t}{q,-ctq^{\alpha+k+1}}L_n^{(\beta)}(cq^k;q)
\frac{q^{(\beta+1)k}}{(-cq^k;q)_\infty}\nonumber\\[0.2cm]
&&\hspace{3cm}=\frac{\left(-tq^{\alpha-\beta}\right)^nq^{\binom n 2}
(q,-cq^{\beta+1},-q^{-\beta}/c;q)_\infty
(q^{\beta+1};q)_n}
{q^n(t,q^{\beta+1},-c,-q/c;q)_\infty(q,q^{\alpha+1};q)_n}
\qhyp11{q^{\alpha-\beta}}{q^{\alpha+n+1}}{q,tq^n}.
\end{eqnarray*}
\begin{proof}
This follows using
(\ref{qLag2})
with (\ref{bilateralorthogqLag}).
\end{proof}
\end{cor}

\begin{cor}\label{cor:32}
Let $n\in\mathbb N_0,$ $|q|<1$, $\alpha,\beta\in\left(-1,\infty\right)$, $\gamma\in\mathbb C$,
$|t|<1-q$, $c>0$. Then
\begin{eqnarray*}
&&\sum_{k=-\infty}^\infty
\qhyp12{\gamma}{q^{\alpha+1},\gamma t}{q,-ctq^{\alpha+k+1}}L_n^{(\beta)}(cq^k;q)
\frac{q^{(\beta+1)k}}{(-cq^k;q)_\infty}\nonumber\\[0.2cm]
&&\hspace{3cm}=\frac{\left(tq^{\alpha-\beta}\right)^n
(t,q,-cq^{\beta+1},-q^{-\beta}/c;q)_\infty
(\gamma,q^{\beta+1};q)_n}
{q^n(\gamma t,q^{\beta+1},-c,-q/c;q)_\infty(q,q^{\alpha+1};q)_n}
\qhyp21{q^{\alpha-\beta},\gamma q^n}{q^{\alpha+n+1}}{q,t}.
\end{eqnarray*}
\begin{proof}
This follows using
(\ref{qLag3})
with (\ref{bilateralorthogqLag}).
\end{proof}
\end{cor}

\subsection{$q$-Integrals}

\subsubsection{$q$-Laguerre polynomials}

One type of orthogonality for the $q$-Laguerre polynomials is
\cite[(14.21.4)]{Koekoeketal}
\begin{equation}
\label{JacksonqorthogqLag}
\int_0^\infty L_m^{(\alpha)}(x;q) L_n^{(\alpha)}(x;q)\frac{x^\alpha}{(-x;q)_\infty}d_qx
=\frac{(1-q)(q,-q^{\alpha+1},-q^{-\alpha};q)_\infty(q^{\alpha+1};q)_n}
{2q^n(q^{\alpha+1},-q,-q;q)_\infty(q;q)_n}\delta_{m,n}.
\end{equation}
Using this orthogonality relation we can obtain new $q$-integrals using our
generalized generating functions for $q$-Laguerre polynomials.

\begin{cor}\label{cor:33}
Let $n\in\mathbb N_0,$ $|q|<1$, $\alpha,\beta\in\left(-1,\infty\right),$
$|t|<(1-q^{\alpha+1})(1-q)$. Then
\begin{eqnarray*}
&&\int_0^\infty \qhyp01{-}{q^{\alpha+1}}{q,-xtq^{\alpha+1}}
L_n^{(\beta)}(x;q) \frac{x^\beta}{(-x;q)_\infty}d_qx\nonumber\\[0.2cm]
&&\hspace{3cm}=\frac{(1-q)\left(tq^{\alpha-\beta}\right)^n(t,q,-q^{\beta
+1},-q^{-\beta};q)_\infty(q^{\beta+1};q)_n}
{2q^n(q^{\beta+1},-q,-q;q)_\infty(q,q^{\alpha+1};q)_n}
\qhyp21{q^{\alpha-\beta},0}{q^{\alpha+n+1}}{q,t}.\nonumber
\end{eqnarray*}
\end{cor}
\begin{proof}
Using (\ref{qLag1})
with (\ref{JacksonqorthogqLag}) completes this proof.
\end{proof}

\begin{cor}\label{cor:34}
Let $n\in\mathbb N_0,$ $|q|<1$, $\alpha,\beta\in\left(-1,\infty\right),$
$|t|<(1-q^{\alpha+1})(1-q)$. Then
\begin{eqnarray*}
&\int_0^\infty \qhyp02{-}{q^{\alpha+1},t}{q,-xtq^{\alpha+1}}
L_n^{(\beta)}(x;q) \frac{x^\beta}{(-x;q)_\infty}d_qx\nonumber\\[0.2cm]
&=\frac{(1-q)\left(-tq^{\alpha-\beta}\right)^nq^{\binom n 2}
(q,-q^{\beta+1},-q^{-\beta};q)_\infty(q^{\beta+1};q)_n}
{2q^n(t,q^{\beta+1},-q,-q;q)_\infty(q,q^{\alpha+1};q)_n}
\qhyp11{q^{\alpha-\beta}}{q^{\alpha+n+1}}{q,tq^n}.\nonumber
\end{eqnarray*}
\end{cor}
\begin{proof}
Using (\ref{qLag2})
with (\ref{JacksonqorthogqLag}) completes this proof.
\end{proof}

\begin{cor}\label{cor:35}
Let $n\in\mathbb N_0,$ $|q|<1$, $\alpha,\beta\in\left(-1,\infty\right),$
$\gamma\in\mathbb C$, $|t|<1-q$. Then
\begin{eqnarray*}
&&\int_0^\infty \qhyp12{\gamma}{q^{\alpha+1},\gamma t}{q,-xtq^{\alpha+1}}
L_n^{(\beta)}(x;q) \frac{x^\beta}{(-x;q)_\infty}d_qx\nonumber\\[0.2cm]
&&\hspace{3cm}=\frac{(1-q)\left(tq^{\alpha-\beta}\right)^n
(t,q,-q^{\beta+1},-q^{-\beta};q)_\infty
(\gamma,q^{\beta+1};q)_n}
{2q^n(\gamma t,q^{\beta+1},-q,-q;q)_\infty(q,q^{\alpha+1};q)_n}
\qhyp21{q^{\alpha-\beta},\gamma q^n}{q^{\alpha+n+1}}{q,t}.\nonumber
\end{eqnarray*}
\end{cor}
\begin{proof}
Using (\ref{qLag3})
with (\ref{JacksonqorthogqLag}) completes this proof.
\end{proof}

\section*{Acknowledgements}
Much thanks to Hans Volkmer for valuable discussions.
The author R. S. Costas-Santos acknowledges 
financial support by Direcci\'on General de 
Investigaci\'on, Ministerio de Econom\'ia 
y Competitividad of Spain, grant MTM2015-65888-C4-2-P.


\def\cprime{$'$} \def\dbar{\leavevmode\hbox to 0pt{\hskip.2ex \accent"16\hss}d}

\end{document}